\documentclass{jgaa-art}

\usepackage{graphicx}
\usepackage{amsmath}
\usepackage{amssymb}
\usepackage{enumerate}

\def\span{{\rm span}}

\def\om{\omega}

\def\ve{\varepsilon}

\def\L{\mathcal{L}}

\def\({\left(}
\def\){\right)}
\def\[{\left[}
\def\]{\right]}

\newcommand{\be}{\begin{equation}}
\newcommand{\ee}{\end{equation}}
\newcommand{\ben}{\begin{equation*}}
\newcommand{\een}{\end{equation*}}
\newcommand{\bea}{\begin{eqnarray}}
\newcommand{\eea}{\end{eqnarray}}
\newcommand{\bean}{\begin{eqnarray*}}
\newcommand{\eean}{\end{eqnarray*}}

\begin{document}
\HeadingAuthor{Devsi Bantva} 
\HeadingTitle{Hamiltonian chromatic number of block graphs} 
\title{Hamiltonian chromatic number of block graphs} 

\author[ddb]{Devsi Bantva}{devsi.bantva@gmail.com} 
\affiliation[ddb]{Lukhdhirji Engineering College, Morvi 363 642 \\ Gujarat (India)} 


\maketitle


\begin{abstract} Let $G$ be a simple connected graph of order $n$. A hamiltonian coloring $c$ of a graph $G$ is an assignment of colors (non-negative integers) to the vertices of $G$ such that $D(u, v)$ + $|c(u) - c(v)|$ $\geq$ $n - 1$ for every two distinct vertices $u$ and $v$ of $G$, where $D(u, v)$ denotes the detour distance between $u$ and $v$ in $G$ which is the length of the longest path between $u$ and $v$. The value \emph{hc(c)} of a hamiltonian coloring $c$ is the maximum color assigned to a vertex of $G$. The hamiltonian chromatic number, denoted by $hc(G)$, is min\{$hc(c)$\} taken over all hamiltonian coloring $c$ of $G$. In this paper, we give a necessary and sufficient condition to achieve a lower bound for the hamiltonian chromatic number of block graphs given in \cite[Theorem 1]{Bantva1}. We present an algorithm for optimal hamiltonian coloring of a special class of block graphs, namely $SDB(p/2)$ block graphs. We characterize level-wise regular block graphs and extended star of blocks achieving this lower bound.
\end{abstract}

\Body 


\section{Introduction}\label{sec:intro}

\textbf{Motivation.} The notion of graph colorings deals with the assignment of non-negative integers to vertices or edges or both of a graph $G$ according to certain rule. A radio $k$-colorings of a graph is the one, which is motivated by the channel assignment problem. In the channel assignment problem, we seek to assign channels to each tv or radio transmitters located at different places such that it satisfies interference constraints. The interference between two transmitters is closely related to the distance between them. It is observed that closer the transmitters then higher the interference; hence the different level of interference occur according to distance between transmitters. In a graph model of channel assignment problem the transmitters are represented by vertices of graph and interference constraints is imposed on edges of a graph. Motivated through this Chartrand \emph{et al.}\cite{Chartrand} introduced the concept of radio $k$-colorings of graphs defined as follows: For a connected graph $G$ of diameter $d$ and an integer $k$ with $1 \leq k \leq d$, a \emph{radio $k$-coloring} of $G$ is an assignment of colors (non-negative integers) to the vertices of $G$ such that $d(u,v)+|c(u)-c(v)| \geq 1+k$ for every two distinct vertices $u$ and $v$ of $G$. The value $rc_{k}(c)$ of a radio $k$-coloring $c$ of $G$ is the maximum color assigned to a vertex of $G$; while the \emph{radio $k$-chromatic number} $rc_{k}(G)$ of $G$ is min\{$rc_{k}(G)$\} taken over all radio $k$-colorings $c$ of $G$. In particular, for $k$ = $d-1$ antipodal vertices can be colored the same and due to this reason, radio $(d-1)$-coloring is called radio antipodal coloring or simply antipodal coloring.

While studying antipodal colorings for paths $P_{n}$ whose inequality is
\be\label{eqn:anti} d(u,v)+|c(u)-c(v)| \geq d, \ee
Chartrand et al.\cite{Chartrand2} observed that in case of paths $P_{n}$, $d$ = $n-1$ and $d(u,v)$ is same as the the length of a longest $u-v$ path which is denoted by $D(u,v)$ known as detour distance between $u$ and $v$, then \eqref{eqn:anti} is equivalent to
\ben\label{eqn:anti1} D(u,v)+|c(u)-c(v)| \geq n-1. \een

Motivated through this they suggested extension for arbitrary graph $G$ and introduced the concept of hamiltonian coloring of graphs which is defined as follows:

\begin{definition} A \emph{hamiltonian coloring} $c$ of a graph $G$ of order $n$ is an assignment of colors (non-negative integers) to the vertices of $G$ such that for every two distinct vertices $u$ and $v$ of $G$ the following holds.
\be\label{eqn:hcdef} D(u, v) + |c(u) - c(v)| \geq n - 1. \ee
The value of $hc(c)$ of a hamiltonian coloring $c$ is the maximum color assigned to a vertex of $G$. The \emph{hamiltonian chromatic number} $hc(G)$ of $G$ is min\{$hc(c)$\} taken over all hamiltonian colorings $c$ of $G$.
\end{definition}

Chartrand \emph{et al.}\cite{Chartrand2} found that two vertices $u$ and $v$ can be assigned the same color only if $G$ contains a hamiltonian $u-v$ path (a $u-v$ path which contains every vertex of graph $G$). Moreover, if $G$ is a hamiltonian-connected graph (a graph which contains a hamiltonian paths for every pair of vertices of it) then all the vertices can be assigned the same color. Thus, they notice that in a certain sense, the hamiltonian chromatic number of a connected graph $G$ measures how close $G$ is to being hamiltonian-connected; less the hamiltonian chromatic number of a connected graph $G$ is, the closer $G$ is to being hamiltonian-connected.

Note that any optimal hamiltonian coloring always assign label 0 to some vertex, then the \emph{span} of any hamiltonian coloring $c$ which is defined as max\{$|c(u) - c(v)|$ : $u, v \in V(G)$\}, is the maximum integer used for coloring. However, in \cite{Chartrand2,Chartrand3,Shen} only positive integers are used as colors. Therefore, the hamiltonian chromatic number defined in this article is one less than that defined in \cite{Chartrand2,Chartrand3,Shen}.

\textbf{Related work.} At present, the hamiltonian chromatic number is known only for handful of graph families. Chartrand \emph{et al.} investigated the exact hamiltonian chromatic numbers for complete graph $K_{n}$, cycle $C_{n}$, star $K_{1,n}$ and complete bipartite graph $K_{r,s}$ in \cite{Chartrand2}. They proved that for every two integers $k$ and $n$ with $k \geq 1$ and $n \geq 3$, there exist a hamiltonian graph (a graph $G$ is called hamiltonian if it has a hamiltonian close path) of order $n$ with hamiltonian chromatics number $k$ if and only if $1 \leq k \leq n$. They also gave an upper bound for the hamiltonian chromatic number of a connected graph in terms of its order. In \cite{Chartrand3}, Chartrand et al. shown that if there exist a hamiltonian coloring of a connected graph $G$ of order $n \geq 3$ such that at least $(n+2)/2$ vertices of $G$ are colored the same, then $G$ is hamiltonian. An upper bound for $hc(P_{n})$ was established by Chartrand \emph{et al.} in \cite{Chartrand2} but the exact value of $hc(P_{n})$ which is equal to the radio antipodal number $ac(P_{n})$ was determined by Khennoufa and Togni in \cite{Khennoufa}. In \cite{Shen}, Shen \emph{et al.} have discussed the hamiltonian chromatic number for graphs $G$ with max\{$D(u,v)$ : $u,v \in V(G)$, $u \neq v$\} $\leq$ $n/2$, where $n$ is the order of graph $G$ and they gave the hamiltonian chromatic number for a special class of caterpillars and double stars. The hamiltonian chromatic number of block graphs and trees is discussed by Bantva in \cite{Bantva1} and \cite{Bantva2}, respectively. The researchers emphasize that determining the hamiltonian chromatic number is interesting but a challenging task even for some basic graph families.

\textbf{Our contribution.} In this paper, we give a necessary and sufficient condition to achieve a lower bound for the hamiltonian chromatics number of block graphs given in \cite[Theorem 1]{Bantva1}. To derive this necessary and sufficient condition, we use an approach similar to one used in \cite{Bantva3,Bantva4}. We also give two other sufficient conditions to achieve this lower bound for the hamiltonian chromatic number of block graphs. We provide an algorithm for the optimal hamiltonian coloring of a special class of block graphs, namely $SDB(n/2)$ block graphs (defined later in the section \ref{sec:main}). As an illustration, we present two class of block graphs, namely level-wise regular block graphs which is more general than symmetric block graphs given in \cite{Bantva1} and extended star of blocks (for definition and detail about both graph families see section \ref{sec:illu}) achieving this lower bound.

\section{Preliminaries}\label{sec:pre}

In this section, we define terms and notations which are necessary for present work. Further, the standard graph theoretic terminology and notation not defined here are used in the sense of \cite{West}.

We consider graph $G$ to be a finite, connected and undirected graph without loops and multiple edges. For a graph $G$, we denote its \emph{vertex set} and \emph{edge set} by $V(G)$ and $E(G)$. The \emph{order} of a graph $G$ is the number of vertices in $G$. For a vertex $v \in V(G)$, the \emph{open neighborhood of $v$} denoted by $N(v)$, is the set of vertices adjacent to $v$. The \emph{distance} $d(u,v)$ between two vertices $u$ and $v$ is the length of a shortest path connecting them. A \emph{cut vertex} of a graph $G$ is a vertex whose deletion increases the number of components of $G$. A \emph{block} of a graph $G$ is a maximal connected subgraph of $G$ that has no cut-vertex. The detour distance between $u$ and $v$, denoted by $D(u,v)$, is the length of the longest path between $u$ and $v$ in $G$. The \emph{detour diameter}, denoted by $\mathcal{D}$, is max\{$D(u,v)$ : $u, v \in V(G)$\}. The \emph{detour eccentricity} $\ve_{D}(v)$ of a vertex $v$ is the detour distance from $v$ to a vertex farthest from $v$. The \emph{detour center} $C_{D}(G)$ of $G$ is the subgraph of $G$ induced by the vertex/vertices of $G$ whose detour eccentricity is minimum. In \cite{Chartrand1}, Chartrand \emph{et al.} shown the following.
\begin{proposition}\label{prop1} The detour center $C_{D}(G)$ of every connected graph $G$ lies in a single block of $G$.
\end{proposition}

A \emph{block graph} is a connected graph all of whose blocks are cliques. We prove the following result about the detour center of block graphs.

\begin{lemma}\label{lem1} The detour center $C_{D}(G)$ of a block graph $G$ is either a vertex or a block.
\end{lemma}
\begin{proof} Let $G$ be a block graph and $|C_{D}(G)| \neq 1$ then by Proposition \ref{prop1}, we have $C_{D}(G) \subseteq B$ for some block $B$. Now if possible then assume that $C_{D}(G) \neq B$ then there exist a vertex $x \in B \setminus C_{D}(G)$. Note that $m = \ve_{D}(x) > \ve_{D}(y)$ for any $y \in C_{D}(G)$ such that $x,y,z_{1},z_{2},...,z_{l},z_{l+1},...,z_{m-1}$ is detour eccentric path of $x$. Let $x,y,z_{1},z_{2},...,z_{l}$ are the vertices of block $B$. Then consider the detour path $y,x,z_{1},z_{2},...,z_{l},z_{l+1},...,z_{m-1}$ for $y$ which gives $\ve_{D}(y) = m$ which contradicts with our assumption that $m = \ve_{D}(x) > \ve_{D}(y)$. Hence, we obtain, $C_{D}(G) = B$.
\end{proof}

In the present work, we denote $|C_{D}(G)|$ = $\omega$ then by Lemma \ref{lem1}, it is clear that $\omega \geq 1$. The vertex/vertices of detour center $C_{D}(G)$ are called \emph{detour central vertex/vertices} for graph $G$. Moreover, we denote the central vertex by $w$ when $|C_{D}(G)|$ = 1 and \{$w^{0},w^{1},...,w^{\om-1}$\} when $|C_{D}(G)| = \om$. For a block graph $G$, define \emph{detour level function} $\mathcal{L}$ on $V(G)$ by

\begin{center}
$\mathcal{L}(u)$ := min\{$D(w, u)$ : $w \in V(C_{D}(G))$\}, for any $u$ $\in$ $V(G)$.
\end{center}
The \emph{total detour level of a graph} $G$, denoted by $\mathcal{L}(G)$, is defined as

\begin{equation*}\label{total:l}
\mathcal{L}(G) := \displaystyle \sum_{u \in V(G)} \mathcal{L}(u).
\end{equation*}

In a block graph $G$, if $u$ is on the $x-v$ path, where $x$ is the nearest detour central vertex for $v$, then $u$ is an \emph{ancestor} of $v$, and $v$ is a \emph{descendent} of $u$. If $u$ is an ancestor and adjacent to $v$ then $u$ is called \emph{parent} of $v$ and $v$ is called \emph{child} of $u$. In a block graph $G$, a block $B$ is called an \emph{ancestor block} of another block $B^{'}$ if the $x-v$ path, where $x \in V(C_{D}(G))$ and $v \in V(B^{'})$ consists a vertex $u \in V(B) \setminus V(B^{'})$. Let $B$ be a block attached to a central vertex. Then the subgraph induced by $B$ and all its descendent blocks is called a \emph{branch} at $w$. Two branches are called \emph{different} if they are induced by two different blocks attached to the same central vertex, and \emph{opposite} if they are induced by two different blocks attached to different central vertices.

For any $u, v \in V(G)$, define $\phi_{D}(u,v)$ := max\{$\mathcal{L}(t)$ : $t$ is a common ancestor of $u$ and $v$\}, and
\begin{eqnarray*}
\delta_{D}(u,v) & = &\left\{
\begin{array}{ll}
\omega-1, & \mbox{if $|C_{D}(G)|$ = $\omega \geq 2$ and detour path $D(u,v)$ passes} \\
& \mbox{through $C_{D}(G)$}, \\
0, & \mbox{otherwise}.
\end{array}
\right.
\end{eqnarray*}

\begin{lemma}\label{lem:phi} Let $G$ be a block graph with $diam(G) \geq 2$. Then for any $u, v \in V(G)$ the following holds:
\begin{enumerate}[\rm (a)]
\item $\phi_{D}(u,v) \geq 0$;
\item $\phi_{D}(u,v)$ = 0 if and only if $u$ and $v$ are in different or opposite branches;
\item $\delta_{D}(u,v)$ = $\omega-1$ if and only if $G$ has two or more detour central vertices and $u$ and $v$ are in opposite branches;
\end{enumerate}
\end{lemma}
\begin{proof} The proof is straight forward from the definition of $\phi_{D}(u,v)$, $\delta_{D}(u,v)$ and different and opposite branches in a block graph $G$. \end{proof}

Note that the detour distance between any two vertices $u$ and $v$ in a block graph $G$ can be given as
\begin{equation}\label{eqn:duv}
D(u,v) = \mathcal{L}(u) + \mathcal{L}(v) - 2 \phi_{D}(u,v) + \delta_{D}(u,v).
\end{equation}

Moreover, equality holds in (\ref{eqn:duv}) if $u$ and $v$ are in different branches when $\omega$ = 1 and in opposite branches when $\omega \geq 2$.

Define $\xi$ = min\{$|V(B_{i})|-1$ : $B_{i}$ is a block attached to detour central vertex\} when $\omega$ = 1; otherwise $\xi$ = 0.

\section{Hamiltonian chromatic number of block graphs}\label{sec:main}

Note that any hamiltonian coloring $c$ of a block graph $G$ is injective if $G$ has three or more branches as in this case no two vertices of block graph contain hamiltonian path. Therefore, throughout this discussion we assume block graphs with three or more branches. Note that a hamiltonian coloring $c$ on $V(G)$, induces an ordering of $V(G)$, which is a line up of the vertices with increasing images. We denote this ordering by $V(G)$ = \{$u_{0}$, $u_{1}$, $u_{2}$, ..., $u_{n-1}$\} with

0 = $c(u_{0})$ $<$ $c(u_{1})$ $<$ $c(u_{2})$ $<$ ... $<$ $c(u_{n-1})$.

Notice that, $c$ is a hamiltonian coloring, then the span of $c$ is $c(u_{n-1})$.

In \cite{Bantva1}, Bantva gave a lower bound for the hamiltonian chromatics number of block graphs as stated in following Theorem.

\begin{theorem}\cite{Bantva1}\label{hc:lower} Let $G$ be a block graph of order $n$ and $\omega$, $\xi$ and $\mathcal{L}(G)$ are defined as earlier. Then
\begin{equation}\label{eqn:lower}
hc(G) \geq (n - 1)(n-\omega) - 2 \mathcal{L}(G) + \xi.
\end{equation}
\end{theorem}

In next result, we give a necessary and sufficient condition to achieve the lower bound for the hamiltonian chromatic number of block graphs given in Theorem \ref{hc:lower}.

\begin{theorem}\label{hc:main} Let $G$ be a block graph of order $n$ and $\omega$, $\xi$ and $\mathcal{L}(G)$ are defined as earlier. Then
\begin{equation}\label{eqn:main}
hc(G) = (n-1)(n-\omega)-2\L(G)+\xi
\end{equation}
holds if and only if there exists an ordering \{$u_{0},u_{1},...,u_{n-1}$\} of the vertices of $G$, with $\L(u_{0})$ = 0, $\L(u_{n-1})$ = $\xi$ when $|C_{D}(G)|$ = 1 and $\L(u_{0})$ = $\L(u_{n-1})$ = 0 when $|C_{D}(G)| \geq 2$, such that
\begin{equation}\label{eqn:Dij}
D(u_{i},u_{j}) \geq \displaystyle\sum_{t=i}^{j-1}(\L(u_{t})+\L(u_{t+1}))-(j-i)(n-\omega)+(n-1), 0\leq i < j \leq n-1.
\end{equation}
Moreover, under this condition the mapping $c$ defined by
\begin{equation} \label{f0}
c(u_{0}) = 0
\end{equation}
\begin{equation}\label{f1}
c(u_{i+1}) = c(u_{i}) + n - \omega - \L(u_{i+1})-\L(u_{i}), 0 \leq i \leq n-2.
\end{equation}
is an optimal hamiltonian coloring of $G$.
\end{theorem}
\begin{proof}\textsf{Necessity:} Suppose that \eqref{eqn:main} holds. Let $c$ be an optimal hamiltonian coloring of $G$ then $c$ induces an ordering of vertices say, $0 = c(u_{0}) < c(u_{1}) < ... < c(u_{n-1})$. The span of $c$ is $\span(c)$ = $hc(G)$ = $(n-1)(n-\om)-2\L(G)+\xi$. Note that this is possible if equality holds in \eqref{eqn:hcdef} together with (a) $\phi_{D}(u_{i},u_{i+1}) = \delta_{D}(u_{i},u_{i+1}) = 0$ and $\L(u_{0})$ = 0, $\L(u_{n-1})$ = $\xi$ when $|C_{D}(G)|$ = 1 and (b) $\phi_{D}(u_{i},u_{i+1})$ = 0, $\delta_{D}(u_{i},u_{i+1})$ = $\om-1$ and $\L(u_{0}) = \L(u_{n-1}) = 0$ when $|C_{D}(G)| \geq 2$ in \eqref{eqn:duv}. This turn the definition of hamiltonian coloring $c(u_{i+1})-c(u_{i})$ = $n-1-D(u_{i},u_{i+1})$ to $c(u_{0})$ = 0 and $c(u_{i+1})$ = $c(u_{i}) + n - \om - \L(u_{i+1}) - \L(u_{i}), 0 \leq i \leq n-2$. The span of $c$ is right-hand side of \eqref{hc:main} and hence $c$ is an optimal hamiltonian coloring.

Moreover, for any two vertices $u_{i}$ and $u_{j}$ (without loss of generality, assume $j > i$), summing the latter definition for index $i$ to $j$, we have
\be\label{eqn:cij} c(u_{j}) - c(u_{i}) = \displaystyle\sum_{t=i}^{j-1}\(n-\om-\L(u_{t})-\L(u_{t+1})\) \ee
Now $c$ is a hamiltonian coloring so that $c(u_{j})-c(u_{i}) \geq n-1-D(u_{i},u_{j})$. Substituting this in \eqref{eqn:cij}, we get
\begin{equation}
D(u_{i},u_{j}) \geq \displaystyle\sum_{t=i}^{j-1}(\L(u_{t})+\L(u_{t+1}))-(j-i)(n-\omega)+(n-1), 0\leq i < j \leq n-1.
\end{equation}
\textsf{Sufficiency:} Suppose that an ordering \{$u_{0},u_{1},...,u_{n-1}$\} of vertices of $G$ satisfies \eqref{eqn:Dij}, and $c$ is defined by \eqref{f0} and \eqref{f1} together with $\L(u_{0})$ = 0 and $\L(u_{n-1}) = \xi$ when $|C_{D}(G)|$ = 1 and $\L(u_{0}) = \L(u_{n-1}) = 0$ when $|C_{D}(G)| \geq 2$. Note that it is enough to prove that $c$ is a hamiltonian coloring with span equal to the right-hand side of \eqref{eqn:main}. Let $u_{i}$ and $u_{j}$ ($0 \leq i < j \leq n-1$) be two arbitrary vertices then by \eqref{f1} and using \eqref{eqn:Dij}, we have
\bean c(u_{j}) - c(u_{i}) & = & (j-i)(n-\om) - \displaystyle\sum_{t=i}^{j-1}\(\L(u_{t})+\L(u_{t+1})\) \\
& \geq & n-1-D(u_{i},u_{j})
\eean
hence $c$ is a hamiltonian coloring.

The span of $c$ is given by
\bean
\span(c) & = & c(u_{n-1})-c(u_{0}) \\
& = & \displaystyle\sum_{t=0}^{n-2}\(c(u_{t})-c(u_{t+1})\) \\
& = & \displaystyle\sum_{t=0}^{n-2}\(n-\om-\L(u_{t+1})-\L(u_{t})\) \\
& = & (n-1)(n-\om)-2\displaystyle\sum_{u \in V(G)}\L(u) + \L(u_{0})+\L(u_{n-1}) \\
& = & (n-1)(n-\om)-2\L(G)+\xi.
\eean
\end{proof}

The following results gives sufficient condition with optimal hamiltonian coloring for the equality in (\ref{eqn:lower}).

\begin{theorem}\label{hc:c1} Let $G$ be a block graph of order $n$ and, $\omega$, $\xi$ and $\mathcal{L}(G)$ are defined as earlier. Then
\begin{equation}\label{eqn:c1}
hc(G) = (n-1)(n-\omega) - 2 \mathcal{L}(G) + \xi,
\end{equation}
if there exists an ordering \{$u_{0}$, $u_{1}$,...,$u_{n-1}$\} with 0 = $c(u_{0}) < c(u_{1}) < ... < c(u_{n-1})$ of vertices of block graph $G$ such that
\begin{enumerate}[\rm (a)]
\item $\mathcal{L}(u_{0})$ = 0, $\mathcal{L}(u_{n-1})$ = $\xi$ when $|C_{D}(G)|$ = 1 and $\mathcal{L}(u_{0})$ = $\mathcal{L}(u_{n-1})$ = 0 when $|C_{D}(G)| \geq 2$,
\item $u_{i}$ and $u_{i+1}$ are in different branches when $|C_{D}(G)| = 1$ and opposite branches when $|C_{D}(G)| \geq 2$,
\item $D(u_{i}, u_{i+1}) \leq n/2$, for $0 \leq i \leq n-2$.
\end{enumerate}
Moreover, under these conditions the mapping $c$ defined by \eqref{f0} and \eqref{f1} is an optimal hamiltonian coloring of $G$.
\end{theorem}
\begin{proof} Suppose (a), (b) and (c) hold for an ordering \{$u_{0}$, $u_{1}$,...,$u_{n-1}$\} of the vertices of $G$ and $c$ is defined by (\ref{f0}) and (\ref{f1}). By Theorem \ref{hc:main}, it is enough to prove that an ordering \{$u_{0}$, $u_{1}$,...,$u_{n-1}$\} satisfies condition (\ref{eqn:Dij}).

Let $u_{i},u_{j} \in V(G), 0 \leq i < j \leq n-1$. Without loss of generality, we assume that $j-i \geq 2$ and for simplicity let the right-hand side of (\ref{eqn:Dij}) is $\mathcal{S}_{i,j}$. Note that $\L(u_{t})+\L(u_{t+1}) \leq n/2-\omega+1$ as $D(u_{t}, u_{t+1}) \leq n/2$ for $0 \leq t \leq n-2$. Hence, we have
$$
\begin{array}{lll}
\mathcal{S}_{i,j} & = & \displaystyle\sum_{t=i}^{j-1}(\L(u_{t})+\L(u_{t+1}))-(j-i)(n-\omega)+(n-1) \\
& \leq & (j-i)(n/2-\omega+1)-(j-i)(n-\omega)+(n-1) \\
& = & (j-1)(1-n/2)+(n-1) \\
& \leq & 2(1-n/2)+(n-1) \\
& = & 1 = D(u_{i},u_{j}).
\end{array}
$$
which completes the proof.
\end{proof}
\begin{theorem}\label{hc:c2} Let $G$ be a block graph of order $n$, detour diameter $\mathcal{D}$ and $\omega$, $\xi$ and $\mathcal{L}(G)$ are defined as earlier. Then
\begin{equation}\label{eqn:c2}
hc(G) = (n-1)(n-\omega) - 2 \mathcal{L}(G) + \xi
\end{equation}
if there exists an ordering \{$u_{0}$, $u_{1}$,...,$u_{n-1}$\} with 0 = $c(u_{0}) < c(u_{1}) < ... < c(u_{n-1})$ of vertices of block graph $G$ such that
\begin{enumerate}[\rm (a)]
\item $\mathcal{L}(u_{0})$ = 0, $\mathcal{L}(u_{n-1})$ = $\xi$ when $|C_{D}(G)|$=1 and $\mathcal{L}(u_{0})$ = $\mathcal{L}(u_{n-1})$ = 0 when $|C_{D}(G)| \geq 2$,
\item $u_{i}$ and $u_{i+1}$ are in different branches when $|C_{D}(G)| = 1$ and opposite branches when $|C_{D}(G)| \geq 2$,
\item $\mathcal{D} \leq n/2$.
\end{enumerate}
Moreover, under these conditions the mapping $c$ defined by \eqref{f0} and \eqref{f1} is an optimal hamiltonian coloring of $G$.
\end{theorem}
\begin{proof} The proof is straight forward as $D(u,v) \leq \mathcal{D} \leq n/2$, for any $u, v \in V(G)$.
\end{proof}

We say a block graph $G$ is star shaped block graph if there exist an ordering \{$u_{0},u_{1},...,u_{n-1}$\} of vertices of $G$ such that $u_{i}$ and $u_{i+1}$, $0 \leq i \leq n-2$ are in different branches when $|C_{D}(G)| = 1$ and in opposite branches when $|C_{D}(G)| \geq 2$ unless one of them is central vertex. If max\{$D(u,v)$ : $u,v \in V(G)$\} $\leq n/2$ then a block graph $G$ is called a maximum detour distance bound $n/2$ or $DB(n/2)$ block graph. We denote a star shaped block graph with maximum detour distance bound $n/2$ by $SDB(n/2)$ block graph. Then from Theorem \ref{hc:c2} note that the hamiltonian chromatic number of any $SDB(n/2)$ block graphs is equal to a lower bound given by right-hand side of \eqref{eqn:lower}. The following algorithm gives an optimal hamiltonian coloring of $SDB(n/2)$ block graphs.

\begin{algorithm} An optimal hamiltonian coloring of $SDB(n/2)$ block graphs.
\end{algorithm}
\noindent\textbf{Input:} A block graph $G$ of order $n$ which is $SDB(n/2)$. \\
\noindent\textbf{Idea:} Find an ordering of vertices of block graph $G$ which satisfies Theorem \ref{hc:c2} and coloring defined by \eqref{f0}-\eqref{f1} is a hamiltonian coloring whose span is right-hand side of \eqref{eqn:c2}. \\
\noindent\textbf{Initialization:} Start with a central vertex $w$. \\
\noindent\textbf{Iteration:} Defined $c$ : $V(G)$ $\rightarrow$ $\{0,1,2,...\}$ as follows. \\
\noindent\textbf{Step-1:} Let $u_{0}$ = $w \in V(C_{D}(G))$  and $u_{n-1} \in N(w)$ such that the block $B$ for which $u_{n-1} \in B$ is the smallest block attached to central vertex $w$ when $|C_{D}(G)|$ = 1 and $u_{0}, u_{n-1} \in V(C_{D}(G))$ when $|C_{D}(G)| \geq 2$. \\
\noindent\textbf{Step-2:} Choose $u_{i}$ and $u_{i+1}$, where $1 \leq i \leq n-2$ in different branches when $|C_{D}(G)|$ = 1 and in opposite branches when $|C_{D}(G)| \geq 2$. Continue this process till all the vertices get order. Note that such ordering is possible as $G$ is star shaped. Then the ordering \{$u_{0}$, $u_{1}$,...,$u_{n-1}$\} satisfies conditions of Theorem \ref{hc:c2}. \\
\noindent\textbf{Step-3:} Defines $c$ : $V(G)$ $\rightarrow$ \{0,1,2,...\} by $c(u_{0})$ = 0 and $c(u_{i+1})$ = $c(u_{i})+n-\omega-\L(u_{i})-\L(u_{i+1}), 0 \leq i \leq n-2$. \\
\noindent\textbf{Output:} The span of $c$ is span($c$) = $c(u_{n-1})$ = $c(u_{0})+(n-1)(n-\omega)-2\(\displaystyle\sum_{u \in V(G)} \L(u)\)+\xi$ = $(n-1)(n-\omega)-2\L(G)+\xi$.

\section{Hamiltonian chromatic number of some block graphs}\label{sec:illu}

\indent\indent In this section, we determine the hamiltonian chromatic number of some block graphs using Theorem \ref{hc:lower} to \ref{hc:c2}. We continue to use the terminology and notation defined in previous section.

\subsection{hc(G) for level-wise regular block graphs}

Let $k_{i} \geq 1$ and $m_{i} \geq 2$, $1 \leq i \leq r$ be positive integers. Setting $C_{D}(G)$ = \{$w$\} and $C_{D}(G)$ = \{$w^{0},w^{1},...,w^{m}$\}. If $C_{D}(G)$ = \{$w$\} then attach $k_{1}$ blocks of size $m_{1}+1$ to $w$ and then again attach $k_{2}$ blocks of size $m_{2}+1$ to each end vertices of previously attached blocks. Continuing in this way, finally attach $k_{r}$ blocks of size $m_{r}+1$ to each end vertices of previously attached blocks at $r-1$ step. We denote this block graph by $G^{1}_{k_{1}m_{1},...,k_{r}m_{r}}$. If $C_{D}(G)$ = \{$w^{0},w^{1},...,w^{m}$\} then attach $k_{1}$ blocks of size $m_{1}+1$ to each $w^{i}$, $i = 0,1,2,...,m$ and then again attach $k_{2}$ blocks of size $m_{2}+1$ to the end vertices of previously attached blocks. Continuing in this way, finally attach $k_{r}$ blocks of size $m_{r}+1$ to each end vertices of previously attached blocks at $r-1$ step. We denote this block graph by $G^{m+1}_{k_{1}m_{1},...,k_{r}m_{r}}$. The block graphs $G^{1}_{k_{1}m_{1},...,k_{r}m_{r}}$ and $G^{m+1}_{k_{1}m_{1},...,k_{r}m_{r}}$ are known as level-wise regular block graphs. Note that diam$\(G^{1}_{k_{1}m_{1},...,k_{r}m_{r}}\)$ = $2r$ and diam$\(G^{m+1}_{k_{1}m_{1},...,k_{r}m_{r}}\)$ = $2r+1$. In our subsequent discussion for simplicity we denote $G^{1}_{k_{1}m_{1},...,k_{r}m_{r}}$ by $G^{1}$ and $G^{m+1}_{k_{1}m_{1},...,k_{r}m_{r}}$ by $G^{m+1}$ only. Notice that in the present work for $G^{1}$ we will not allow $k_{1}$ = 1 as in this case central vertex of block graph has one branch only.

\begin{theorem}\label{thm:level} Let $m, k_{i} \geq 1$ and $m_{i} \geq 2$ where $1 \leq i \leq r$ be integers. Then
\be\label{hc:level1}
hc(G^{1}) = \[\displaystyle\sum_{i=1}^{r}\(\displaystyle\prod_{1 \leq j \leq i}k_{i}m_{i}\)\]^{2}-2\[\displaystyle\sum_{i=1}^{r}\(\(\displaystyle\prod_{1 \leq j \leq i}k_{i}m_{i}\)\(\displaystyle\sum_{j=1}^{i}m_{j}\)\)\]+m_{1}
\ee
and
\begin{equation*}
hc(G^{m+1}) = \[m+(m+1)\displaystyle\sum_{i=1}^{r}\(\displaystyle\prod_{1 \leq j \leq i}k_{i}m_{i}\)\]\[(m+1)\displaystyle\sum_{i=1}^{r}\(\displaystyle\prod_{1 \leq j \leq i}k_{i}m_{i}\)\]
\end{equation*}
\be\label{hc:level2}
\hspace{5cm}- 2\[\displaystyle\sum_{i=1}^{r}(m+1)\(\(\displaystyle\prod_{1 \leq j \leq i}k_{i}m_{i}\)\(\displaystyle\sum_{j=1}^{i}m_{j}\)\)\]
\ee
\end{theorem}
\begin{proof} \textbf{(1) For $G^{1}$:}~~The order and the total level of $G^{1}$ are given by
\be\label{level:p1} n := 1+\displaystyle\sum_{i=1}^{r}\(\displaystyle\prod_{1 \leq j \leq i}k_{i}m_{i}\) \ee

\be\label{level:LG1} \L(G^{1}) := \displaystyle\sum_{i=1}^{r}\(\(\displaystyle\prod_{1 \leq j \leq i}k_{i}m_{i}\)\(\displaystyle\sum_{j=1}^{i}m_{j}\)\) \ee

Substituting \eqref{level:p1} and \eqref{level:LG1} into \eqref{eqn:lower} gives the right-hand side of \eqref{hc:level1}

Now we give systematic ordering of vertices of $G^{1}$ such that it satisfies conditions of Theorem \ref{hc:c1}. Note that $G^{1}$ has a unique central vertex $w$. Now denote the $k_{1}m_{1}$ children of the $w$ by $w_{0}$, $w_{1}$,...,$w_{k_{1}m_{1}-1}$ such that any $k_{1}$ consecutive vertices are in different blocks with $w_{a}$ and $w_{a+k_{1}}$ are in the same block for $0 \leq a \leq k_{1}m_{1}-k_{1}-1$. Denote the $k_{2}m_{2}$ children of each $w_{i}$ by $w_{i0}$, $w_{i1}$,...,$w_{i(k_{2}m_{2}-1)}$, $0 \leq i \leq k_{1}m_{1}-1$ such that any $k_{2}$ consecutive vertices are in different blocks with $w_{ia}$ and $w_{i(a+k_{2})}$ are in the same block for $0 \leq a \leq k_{2}m_{2}-k_{2}-1$. Inductively, denote the $k_{l+1}m_{l+1}$ children of $w_{i_{1}i_{2}...i_{l}}$ ($0 \leq i_{1} \leq k_{1}m_{1}-1, 0 \leq i_{2} \leq k_{2}m_{2}-1,...,0 \leq i_{l} \leq k_{l}m_{l}-1$) by $w_{i_{1}i_{2}...i_{l}i_{l+1}}$ where $0 \leq i_{l+1} \leq k_{l+1}m_{l+1}-1$ such that any $k_{l+1}$ consecutive vertices are in different blocks with $w_{i_{1}i_{2}...i_{l}a}$ and $w_{i_{1}i_{2}...i_{l}(a+k_{l+1})}$ are in the same block for $0 \leq a \leq k_{l+1}m_{l+1}-k_{l+1}-1$. Continue this until all vertices of $G^{1}$ are indexed in this way. We then rename the vertices of $G^{1}$ as follows:

Let $u_{0}$ = $w$ and for $1 \leq j \leq n-1$, let

$u_{j} := w_{i_{1}i_{2}...i_{l}}$,

where $j=1+i_{1}+i_{2}(k_{1}m_{1})+...+i_{l}\(\displaystyle\prod_{i=1}^{l-1}k_{i}m_{i}\)+\displaystyle\sum_{l+1 \leq t \leq r}\(\displaystyle\prod_{i=1}^{t}k_{i}m_{i}\)$.

Note that $u_{n-1}$ is adjacent to $w$ and for $1 \leq i \leq n-2$, $u_{i}$ and $u_{i+1}$ are in different branches so that $\phi_{D}(u_{i},u_{i+1})$ = 0. Moreover, for $0 \leq i \leq n-1$ we have $D(u_{i},u_{i+1}) \leq n/2$ and hence above defined ordering \{$u_{0}$,$u_{1}$,...,$u_{n-1}$\} of vertices of $G^{1}$ satisfies conditions of Theorem \ref{hc:c1}. The hamiltonian coloring defined by \eqref{f0} and \eqref{f1} is an optimal hamiltonian coloring whose span is, span($c$) = $(n-1)^{2}-2\L(G^{1})+\xi$ which is right-hand side of \eqref{hc:level1} using \eqref{level:p1} and \eqref{level:LG1} in the present case.

\textbf{(2) For $G^{m+1}$:}~~The order and the total level of $G^{m+1}$ are given by
\be\label{level:p2} n := 1+m+(m+1)\displaystyle\sum_{i=1}^{r}\(\displaystyle\prod_{1 \leq j \leq i}k_{i}m_{i}\) \ee

\be\label{level:LGm} \L(G^{m+1}) := \displaystyle\sum_{i=1}^{r}(m+1)\(\(\displaystyle\prod_{1 \leq j \leq i}k_{i}m_{i}\)\(\displaystyle\sum_{j=1}^{i}m_{j}\)\) \ee

Substituting \eqref{level:p2} and \eqref{level:LGm} into \eqref{eqn:lower} gives the right-hand side of \eqref{hc:level2}

Now we give systematic ordering of vertices of $G^{m+1}$ such that it satisfies conditions of Theorem \ref{hc:c1}. Note that in this case $G^{m+1}$ has $m+1$ central vertices $w^{0},w^{1},...,w^{m}$. Now denote the $k_{1}m_{1}$ children of each $w^{t}, 0 \leq t \leq m$ by $w^{t}_{0}$, $w^{t}_{1}$,...,$w^{t}_{k_{1}m_{1}-1}$ such that any $k_{1}$ consecutive vertices are in different blocks with $w^{t}_{a}$ and $w^{t}_{a+k_{1}}$ are in the same block for $0 \leq a \leq k_{1}m_{1}-k_{1}-1$. Denote the $k_{2}m_{2}$ children of each $w^{t}_{i}, 0 \leq t \leq m$ by $w^{t}_{i0}$, $w^{t}_{i1}$,...,$w^{t}_{i(k_{2}m_{2}-1)}$, $0 \leq i \leq k_{1}m_{1}-1$ such that any $k_{2}$ consecutive vertices are in different blocks with $w^{t}_{ia}$ and $w^{t}_{i(a+k_{2})}$ are in the same block for $0 \leq a \leq k_{2}m_{2}-k_{2}-1$. Inductively, denote the $k_{l+1}m_{l+1}$ children of $w^{t}_{i_{1}i_{2}...i_{l}}$ ($0 \leq t \leq m, 0 \leq i_{1} \leq k_{1}m_{1}-1, 0 \leq i_{2} \leq k_{2}m_{2}-1,...,0 \leq i_{l} \leq k_{l}m_{l}-1$) by $w^{t}_{i_{1}i_{2}...i_{l}i_{l+1}}$ where $0 \leq i_{l+1} \leq k_{l+1}m_{l+1}-1$ such that any $k_{l+1}$ consecutive vertices are in different blocks with $w^{t}_{i_{1}i_{2}...i_{l}a}$ and $w^{t}_{i_{1}i_{2}...i_{l}(a+k_{l+1})}$ are in the same block for $0 \leq a \leq k_{l+1}m_{l+1}-k_{l+1}-1$. Continue this until all vertices of $G^{m+1}$ are indexed in this way. We then rename these vertices as follows:

Let $u^{t}_{j} := w^{t}_{i_{1}i_{2}...i_{l}}$,

where $j=1+i_{1}+i_{2}(k_{1}m_{1})+...+i_{l}\(\displaystyle\prod_{i=1}^{l-1}k_{i}m_{i}\)+\displaystyle\sum_{l+1 \leq t \leq r}\(\displaystyle\prod_{i=1}^{t}k_{i}m_{i}\)$.

Define an ordering \{$u_{0}, u_{1},...,u_{n-1}$\} as follows. Let $u_{0}$ = $w^{m}$ and for $1 \leq j \leq n-m-1$, let
$$
u_{j}:=u^{t}_{s} \mbox{ where } s=\lceil j/(m+1)\rceil, \mbox{ if } j \equiv t (\mbox{mod }(m+1)).
$$
For $n-m \leq j \leq n-1$, let
$$
u_{j}:=w^{j-n+m}
$$
Note that $u_{0}, u_{n-1} \in C_{D}(G^{m+1})$, and $u_{i}$ and $u_{i+1}$, $1 \leq i \leq n-1$ are in opposite branches so that $\phi_{D}(u_{i},u_{i+1})$ = 0 and $\delta_{D}(u_{i},u_{i+1})$ = $\omega-1$ = $m$. Moreover, for $0 \leq i \leq n-1$ we have, $D(u_{i},u_{i+1}) \leq n/2$ and hence above defined ordering \{$u_{0}$,$u_{1}$,...,$u_{n-1}$\} of vertices of $G^{m+1}$ satisfies conditions of Theorem \ref{hc:c1}. The hamiltonian coloring defined by \eqref{f0} and \eqref{f1} is an optimal hamiltonian coloring whose span is span($c$) = $(n-1)(n-\om)-2\L(G^{m+1})+\xi$ which is right-hand side of \eqref{hc:level2} using \eqref{level:p2} and \eqref{level:LGm} in the present case.
\end{proof}

Note that a \emph{symmetric block graph}, denoted by $B_{p,k}$(or $B_{p,k}(d)$ if diameter is $d$) is defined in \cite{Bantva1} which is a block graph with at least two blocks such that all blocks are cliques of size $p$, each cut vertex is exactly in $k$ blocks and the eccentricity of end vertices is same. Note that symmetric block graphs $B_{p+1,k+1}$ are level-wise regular block graphs by taking $m$ = $p-1$, $m_{1}$ = $m_{2}$ = ... = $m_{k}$ = $p$ and $k_{1}$ = $k_{2}$ = ... = $k_{k}$ = $k$ in $G^{1}_{k_{1}m_{1},...,k_{r}m_{r}}$ and $G^{m+1}_{k_{1}m_{1},...,k_{r}m_{r}}$. It is straight forward to verify that in this case, Theorem \ref{thm:level} becomes \cite[Theorem 7]{Bantva1} stated as follows.

\begin{theorem}\label{sym:thm} Let $k \geq 1$, $p \geq 2$, $d \geq 3$ be integers, $r$ = $\lfloor d/2 \rfloor$ and $\Phi_{r}(x)$ = $1+x+x^{2}+...+x^{r-1}$. Then $hc\left(B_{p+1,k+1}(d)\right)$
\begin{eqnarray}\label{sym:eqn}
& = & \left\{
\small{
\begin{array}{ll}
p^{2}(k+1)\left[\Phi_{r}(kp)\left((k+1)\Phi_{r}(kp)-2r\right)
+\frac{2\left(\Phi_{r}(kp)-r\right)}{kp-1}\right]+p, & $if $ d $ is even$, \vspace{0.3 cm} \\
kp^{2}(p+1)\left[\Phi_{r}(kp)\left(k(p+1)\Phi_{r}(kp)-2r+1\right)
+\frac{2\left(\Phi_{r}(kp)-r\right)}{kp-1}\right], & $if $ d $ is odd$.
\end{array}}
\right.
\end{eqnarray}
\end{theorem}

It is interesting that $B_{2,k}(2)$ are stars $K_{1,k}$, $B_{p,k}(2)$ are one point union of $k$ complete graphs (a one point union of $k$ complete graphs, also denoted by $K_{p}^{k}$, is a graph obtained by taking $v$ as a common vertex such that any two copies of $K_{p}$ are edge disjoint and do not have any vertex common except $v$), $B_{2,2}(d)$ are paths $P_{d+1}$ and $B_{2,k}(d)$ are symmetric trees (see \cite{Vaidya}). The hamiltonian chromatic number of stars $K_{1,k}$ is reported by Chartrand \emph{et al.} in \cite{Chartrand1}. The hamiltonian chromatic number of $B_{p,k}(2)$ is investigated by Bantva in \cite{Bantva1}. The hamiltonian chromatic number of paths which is equal to the antipodal radio number of paths given by Khennoufa and Togni in \cite{Khennoufa} and the hamiltonian chromatic number of symmetric trees is investigated by Bantva in \cite{Bantva1}.

\subsection{hc(G) for extended stars of blocks}

An extended star, denoted by $S_{m}^{k}$, is a tree obtained by identifying one end vertex of each $k$ copies of path of length $m$. An extended star of blocks, denoted by $S_{m,p}^{k}$, is a block graph obtain by replacing each edge in $S_{m}^{k}$ by complete graph $K_{p}$.
\begin{theorem} Let $m \geq 2$ and $p, k \geq 3$ be integers. Then
\be\label{hc:S}
hc(S_{m,p}^{k}) = km(p-1)^{2}(km-m-1)+p-1.
\ee
\end{theorem}
\begin{proof} The order and the total detour level of $S_{m,p}^{k}$ are given by
\be\label{S:p}
n := km(p-1)+1
\ee
\be\label{S:L}
\L(G) := \frac{km(m+1)(p-1)^{2}}{2}
\ee
Substituting \eqref{S:p} and \eqref{S:L} into \eqref{eqn:lower} gives the right-hand side of \eqref{hc:S}. We now prove that the right-hand side of \eqref{hc:S} is the actual value for the hamiltonian chromatic number of extended star of blocks. Note that for this purpose it is enough to give an ordering of vertices of extended star of block which satisfies conditions of Theorem \ref{hc:c1}.

Let $w$ be the common vertex of all branches and $v_{i,j}^{l}$, $1 \leq i \leq m, 1 \leq j \leq p-1, 1 \leq l \leq k$ is the $j^{th}$ vertex at $i^{th}$ distance from central vertex $w$ in $l^{th}$ branch. Without loss of generality, we assume that $v_{i,p-1}^{l}$, $1 \leq i \leq m-1, 1 \leq i \leq k$ are cut vertices then, to facilitate hamiltonian coloring we find a linear order \{$u_{0}$,$u_{1}$,...,$u_{n-1}$\} of vertices $S_{m,p}^{k}$ as follows: We first set $u_{0}$ = $v_{1,1}^{1}$ and $u_{n-1}$ = $w$. Next, for $u_{t}$, $1 \leq t \leq n-2$, we consider the following cases.

\textsf{Case-1:} $k$ is odd.~~In this case, if $m$ is odd, then for odd $l$, let $u_{t} := v_{i,j}^{l}$,
\begin{eqnarray*}
\mbox{ where } t = &\left\{
\begin{array}{ll}
(i-1)2k(p-1)+(j-1)2k+(l-1), & \mbox{ if } i < \frac{m+1}{2}, \\ [0.3cm]
(i-1)2k(p-1)+(j-1)k+(l-1), & \mbox{ if } i = \frac{m+1}{2}, \\ [0.3cm]
(m-i)2k(p-1)+(j-1)2k+(l-1)+k, & \mbox{ if } i > \frac{m+1}{2}.
\end{array}
\right.
\end{eqnarray*}
and for even $l$, let $u_{t} := v_{i,j}^{l}$,
\begin{eqnarray*}
\mbox{ where } t = &\left\{
\begin{array}{ll}
(i-1)2k(p-1)+(j-1)2k+(l-1)+k, & \mbox{ if } i < \frac{m+1}{2}, \\ [0.3cm]
(i-1)2k(p-1)+(j-1)k+(l-1), & \mbox{ if } i = \frac{m+1}{2}, \\ [0.3cm]
(m-i)2k(p-1)+(j-1)2k+(l-1), & \mbox{ if } i > \frac{m+1}{2}.
\end{array}
\right.
\end{eqnarray*}
If $m$ is even then, for odd $l$, let $u_{t} := v_{i,j}^{l}$,
\begin{eqnarray*}
\mbox{ where } t = &\left\{
\begin{array}{ll}
(i-1)2k(p-1)+(j-1)2k+(l-1), & \mbox{ if } i \leq \frac{m}{2}, \\ [0.3cm]
(m-i)2k(p-1)+(j-1)2k+(l-1)+k, & \mbox{ if } i > \frac{m}{2}.
\end{array}
\right.
\end{eqnarray*}
and for even $l$, let $u_{t} := v_{i,j}^{l}$,
\begin{eqnarray*}
\mbox{ where } t = &\left\{
\begin{array}{ll}
(i-1)2k(p-1)+(j-1)2k+(l-1)+k, & \mbox{ if } i \leq \frac{m}{2}, \\ [0.3cm]
(m-i)2k(p-1)+(j-1)2k+(l-1), & \mbox{ if } i > \frac{m}{2}.
\end{array}
\right.
\end{eqnarray*}

\textsf{Case-2:} $k$ is even.~~In this case, let $u_{t} := v_{i,j}^{l}$,

\begin{eqnarray*}
\mbox{ where } t = &\left\{
\begin{array}{ll}
(i-1)k(p-1)+(j-1)k+(l-1), & \mbox{ if } l \mbox{ is even}, \\ [0.3cm]
(m-i)k(p-1)+(j-1)k+(l-1), & \mbox{ if } l \mbox{ is odd}.
\end{array}
\right.
\end{eqnarray*}

Then above defined ordering \{$u_{0}$,$u_{1}$,...,$u_{n-1}$\} of vertices satisfies Theorem \ref{hc:c1}. The hamiltonian coloring $c$ defined by \eqref{f0} and \eqref{f1} is an optimal hamiltonian coloring. The span of $c$ is span($c$) = $c(u_{n-1})$ = $c(u_{0})+(n-1)^{2}-2\L(S_{m,p}^{k})+p-1$ which is exactly the right-hand side of \eqref{hc:S} by using \eqref{S:p} and \eqref{S:L} in the present case of extended star of blocks which completes the proof.
\end{proof}

\subsection{hc(G) and graph operation}

The next result explain how to construct a larger block graph for which Theorem \ref{hc:main} holds  using graph operation from given graphs satisfying Theorem \ref{hc:main}.

Let $G_{1},G_{2},...,G_{k}$ are block graphs of order $n_{i}$ and $C_{D}(G_{i})$ = \{$w_{i}$\} respectively, where $i$ = $1,2,...,k$. We define $G_{k}(w)$ to be the block graph obtained by identifying a central vertex $w_{i}$ of each $G_{i}$ to a single vertex $w$. It is clear that $C_{D}(G_{k}(w))$ = \{$w$\}. The order of $G_{k}(w)$ is $n$ = $n_{1}+n_{2}+...+n_{k}-k+1$ and $\xi$ = min\{$\xi_{i}$ : $i$ = 1,2,...,$k$\}.
\begin{theorem}\label{thm:Gkw} Let $G_{1},G_{2},...,G_{k}$ are block graphs of order $n_{i}$ and $C_{D}(G_{i})$ = \{$w_{i}$\}, $1 \leq i \leq k$ for which Theorem \ref{hc:main} hold then so is for $G_{k}(w)$ and
\be\label{hc:Gkw}
hc(G_{k}(w)) = \displaystyle\sum_{i=1}^{k} \(hc(G_{i})-\xi_{i}\) + 2 \displaystyle\sum_{1 \leq i < j \leq k} (n_{i}-1)(n_{j}-1) + \xi.
\ee
\end{theorem}
\begin{proof} The order and the total level of $G_{k}(w)$ are given by
\be\label{Gkw:p}
n := n_{1}+n_{2}+...+n_{k}-k+1
\ee
\be\label{Gkw:L}
\L(G_{k}(w)) := \L(G_{1})+\L(G_{2})+...+\L(G_{k})
\ee
Substituting \eqref{Gkw:p} and \eqref{Gkw:L} into \eqref{eqn:lower} gives the right-hand side of \eqref{hc:Gkw}.

Now we give an ordering of vertices of $G_{k}(w)$ which satisfies conditions of Theorem \ref{hc:main}. Without loss of generality, we assume that the smallest block attached to $w$ lies in $G_{k}$. Let $u^{i}_{0}, u^{i}_{1},...,u^{i}_{n_{i}-1}$ be the ordering of vertices of $G_{i}$, ($1 \leq i \leq k$) which satisfies conditions of Theorem \ref{hc:main}. Let $u_{0}$ = $w$ and $u_{n-1}$ = $u^{k}_{n_{k}-1}$ then $u_{n-1} \in N(w)$ and $\L(u_{n-1})$ = $\xi$. For all other $u_{t}$, $1 \leq t \leq n-2$, let \\

$u_{t} := u_{l}^{j}$, where $t = \displaystyle\sum_{t=1}^{j}n_{j-1}+l$, $1 \leq j \leq k$, $1 \leq l \leq n_{j}-1$ and $n_{0}$ = 0.

Then note that $u_{i}$ and $u_{i+1}$ are in different branches for $1 \leq i \leq n-1$. \\
\textbf{Claim:} The ordering \{$u_{0},u_{1},...,u_{n-1}$\} satisfies \eqref{eqn:Dij}. \\
Let $u_{i}$ and $u_{j}$, $0 \leq i, j \leq n-1$ be two arbitrary vertices. We consider the following two cases. \\
\textsf{Case-1:} $u_{i}, u_{j} \in V(G_{l})$ for some $1 \leq l \leq k$. \\
If $u_{i}, u_{j} \in V(G_{l})$ for some $1 \leq l \leq k$ then by definition of ordering $u_{i}$ = $u^{l}_{a}$ and $u_{j}$ = $u^{l}_{b}$. Note that $D(u_{i},u_{j})$ = $D(u^{l}_{a},u^{l}_{b})$ and $j-i$ = $a-b$. Moreover, $u_{t} \in V(G_{l}), i \leq t \leq j$ and it is obvious that $n \geq n_{l}$. Let the right-hand side of \eqref{eqn:Dij} is $\mathcal{S}_{i,j}$ then \\
\bean
\mathcal{S}_{i,j} & = & \displaystyle\sum_{t=i}^{j-1}\(\L(u_{t})+\L(u_{t+1})\)-(j-i)(n-1)+(n-1) \\
& \leq & \displaystyle\sum_{t=i}^{j-1}\(\L(u_{t})+\L(u_{t+1})\)-(j-i-1)(n_{l}-1) \\
& = & \displaystyle\sum_{t=m}^{n-1}\(\L(u^{l}_{t})+\L(u^{l}_{t+1})\) - (a-b-1)(n_{l}-1) \\
& \leq & D(u^{l}_{a}, u^{l}_{b}) \\
& = & D(u_{i}, u_{j})
\eean
\textsf{Case-2:} $u_{i} \in V(G_{l_{1}})$ and $u_{j} \in V(G_{l_{2}})$ for some $1 \leq l_{1} \neq l_{2} \leq k$. \\
If $u_{i} \in V(G_{l_{1}})$ and $u_{j} \in V(G_{l_{2}})$ for some $1 \leq l_{1} \neq l_{2} \leq k$ then by definition of ordering $D(u_{i},u_{j})$ = $\L(u_{i})+\L(u_{j})$. Let $\alpha$ = max\{$\L(u_{t})$ : $i < t < j$\}. Let the right-hand side of \eqref{eqn:Dij} is $\mathcal{S}_{i,j}$ then
\bean
\mathcal{S}_{i,j} & = & \displaystyle\sum_{t=i}^{j-1}\(\L(u_{t})+\L(u_{t+1})\) - (j-i)(n-1)+(n-1) \\
& = & \L(u_{i}) + \L(u_{j}) + 2 \displaystyle\sum_{t=i+1}^{j-1}\L(u_{t}) - (j-i-1)(n-1) \\
& \leq & \L(u_{i}) + \L(u_{j}) - (j-i-1)(n-1-2\alpha) \\
& \leq & \L(u_{i}) + \L(u_{j}) \\
& = & D(u_{i}, u_{j})
\eean
Thus, from Case 1 and 2, we obtained an ordering \{$u_{0},u_{1},...,u_{n-1}$\} satisfies \eqref{eqn:Dij}. Hence the hamiltonian coloring defined by \eqref{f0} and \eqref{f1} is an optimal hamiltonian coloring. The hamiltonian chromatic number of $G_{k}(w)$ is given by
\bean
hc(G_{k}(w)) & = & (n-1)^{2}-2 \L(G_{k}(w)) + \xi \\
& = & \(\displaystyle\sum_{i=1}^{k}(n_{i}-k)\)^{2}-2\displaystyle\sum_{i=1}^{k}\L(G_{i})+\xi \\
& = & \displaystyle\sum_{i=1}^{k}(n_{i}-1)^{2}+2\displaystyle\sum_{1 \leq i < j \leq k}(n_{i}-1)(n_{j}-1)-2\displaystyle\sum_{i=1}^{k}\L(G_{i})+\xi \\
& = & \displaystyle\sum_{i=1}^{k}\[(n_{i}-1)^{2}-2\L(G_{i})\]+2\displaystyle\sum_{1 \leq i < j \leq k}(n_{i}-1)(n_{j}-1)+\xi \\
& = & \displaystyle\sum_{i=1}^{k}\(hc(G_{i})-\xi_{i}\)+2\displaystyle\sum_{1 \leq i < j \leq k}(n_{i}-1)(n_{j}-1)+\xi.
\eean
\end{proof}

\section{Conclusion}

In this paper, we presented a necessary and sufficient condition together with an optimal hamiltonian coloring to achieve a lower bound for the hamiltonian chromatic number of block graphs given in \cite[Theorem 1]{Bantva1}. Note that our necessary and sufficient condition deals only with an ordering of vertices without assigning actual colors (non-negative integers) to vertices which is useful to determine whether certain block graphs has the hamiltonian chromatic number is equal to a lower bound given in \cite[Theorem 1]{Bantva1} or not. We also gave two other sufficient conditions to achieve this lower bound. We presented an algorithm for the hamiltonian chromatic number of a special class of block graphs, namely $SDB(n/2)$. We also determined the hamiltonian chromatic number of level-wise regular block graphs which is more general case than symmetric block graph (given in \cite{Bantva1}) and extended star of blocks using our main Theorem \ref{hc:main}. Further, readers are suggested to find more classes of block graphs for which Theorem \ref{hc:main} hold.


\section*{Acknowledgements}

The author is grateful to anonymous referee for their valuable comments and suggestions which improved the paper.



\clearpage

\begin{thebibliography}{00}
\bibitem{Bantva1} D. Bantva, On Hamiltonian Colorings of Block graphs, In: Kaykobad, M., Petrechi, R., (eds.) WALCOM: Algorithms and Computation, LNCS: 9627, 28-39, 2016.

\bibitem{Bantva2} D. Bantva, On Hamiltonian Colorings of Trees, In: Govindrajan, S., Maheshwari, A.,(eds.) Algorithms and Discrete Applied Mthematics, CALDAM 2016, LNCS: 9602, 49-60, 2016.
    
\bibitem{Bantva3} D. Bantva and S. Vaidya and S. Zhou, Radio number of trees, Electron. Notes in Discrete Math., 48, 135-141, 2015.

\bibitem{Bantva4} D. Bantva and S. Vaidya and S. Zhou, Radio number of trees, Discrete Applied Math., 217, 110-121, 2017.

\bibitem{Chartrand} G. Chartrand and D. Erwin and P. Zhang, A graph labeling problem suggested by FM channel restrictions, Bull. Inst. Combin. Appl., 43, 43-57, 2005.
     
\bibitem{Chartrand1} G. Chartrand and H. Escuadro and P. Zhang, Detour distance in graphs, J. Combin. Math. Combin. Comput., 53, 75-94, 2005.

\bibitem{Chartrand2} G. Chartrand and L. Nebesk$\acute{y}$ and P. Zhang, Hamiltonian coloring of graphs, Discrete Applied Math., 146, 257-272, 2005.

\bibitem{Chartrand3} G. Chartrand and L. Nebesk$\acute{y}$ and P. Zhang, On hamiltonian colorings of graphs, Discrete Math., 290, 133-143, 2005.

\bibitem{Khennoufa} R. Khennoufa and O. Togni, A note on radio antipodal colourings of paths, Mathematica Bohemica, 130(3), 277-282, 2005.

\bibitem{Shen} Y. Shen and W. He and X. Li and D. He and X. Yang, On hamiltonian colorings for some graphs, Discrete Applied Math., 156, 3028-3034, 2008.

\bibitem{Vaidya} S. Vaidya and D. Bantva, Symmetric regular cacti - properties and enumeration, Proyecciones Journal of Mathematics, 31(3), 261-275, 2012.

\bibitem{West} D. B. West, Introduction to Graph theory, Prentice -Hall of India, 2001.
\end{thebibliography}
\end{document}